\documentclass[12pt]{amsart}
\usepackage{geometry}
\usepackage{amsmath,amsthm,amssymb}
\usepackage[utf8]{inputenc}
\usepackage[T1]{fontenc}

\usepackage{enumerate}
\usepackage{todonotes}
\usepackage{mathrsfs}
\usepackage{hyperref}
\usepackage{nicefrac}
\usepackage{tikz}
\usetikzlibrary{matrix}

\usepackage{url}
\makeatletter
\g@addto@macro{\UrlBreaks}{\UrlOrds}
\makeatother

\usepackage[sort,numbers]{natbib}

\providecommand{\noopsort}[1]{} 

\newtheorem{Th}{Theorem}[section]

\newtheorem{Lemma}[Th]{Lemma}

\newtheorem{Remark}[Th]{Remark}
\newtheorem{Def}{Definition}[section]
\newtheorem{Cor}[Th]{Corollary}

\newcommand{\beq}{\begin{equation}}
\newcommand{\eeq}{\end{equation}}

\def\scalar(#1,#2){(#1\mid#2)}

\newcommand{\cb}{\mathcal{B}}

\newcommand{\cd}{\mathcal{D}}

\newcommand{\zdr}{(Z,\mathcal{D},\rho)}

\newcommand{\ov}{\overline}

\newcommand{\bs}{\mathbb{S}}

\newcommand{\R}{{\mathbb{R}}}
\newcommand{\T}{{\mathbb{T}}}
\newcommand{\C}{{\mathbb{C}}}
\newcommand{\Z}{{\mathbb{Z}}}
\newcommand{\N}{{\mathbb{N}}}
\newcommand{\PP}{{\mathbb{P}}}
\newcommand{\D}{{\mathbb{D}}}

\newcommand{\vep}{\varepsilon}

\newcommand{\mob}{\mu}
\newcommand{\lio}{\lambda}

\newcommand{\bfu}{\boldsymbol{u}}

\title{Rigidity in dynamics and M\"obius disjointness}
\author{Adam Kanigowski}
\address{
  Department of Mathematics, University of Maryland at College Park, College Park,MD 20740, USA\\
E-mail address: adkanigowski@gmail.com
}
\author{Mariusz Lema\'nczyk}
\address{
Faculty of Mathematics and Computer Science, Nicolaus Copernicus University, Chopin street 12/18, 87-100 Toru\'n, Poland\\
E-mail address: mlem@mat.umk.pl
}
\author{Maksym Radziwi\l\l}
\address{
Department of Mathematics, Caltech, 1200 E California BLVD, Pasadena, CA
91125, USA\\
E-mail address: maksym.radziwill@gmail.com
}

\begin{document}
\maketitle
\thispagestyle{empty}
\begin{abstract} Let $(X, T)$ be a topological dynamical system. We show that if each invariant measure of $(X, T)$ gives rise to a measure-theoretic dynamical system that is either:\begin{itemize}
\item[a.] rigid along a sequence of ``bounded prime volume'' or
\item[b.] admits a polynomial rate of rigidity on a linearly dense subset in $C(X)$
\end{itemize}
then $(X, T)$ satisfies Sarnak's conjecture on M\"obius disjointness. We show that the same conclusion also holds if there are countably many invariant ergodic measures, and each of them satisfies a. or b. This recovers some earlier results and implies Sarnak's conjecture in the following new cases: for almost every interval exchange map of $d$ intervals with $d \geq 2$, for $C^{2+\epsilon}$-smooth skew products over rotations and $C^{2+\epsilon}$-smooth flows (without fixed points) on the torus. In particular, these are improvements of earlier results of respectively Chaika-Eskin, Wang and Huang-Wang-Ye. We also discuss some purely arithmetic consequences for the Liouville function.
\end{abstract}
\section{Introduction}
Let $(X,d)$ be a compact metric space and $T:X\to X$ a homeomorphism. We call $(X, T)$ a \textit{topological dynamical system}. Sarnak's conjecture \cite{Sa} asserts that if $(X,T)$ is of zero entropy, then
\beq\label{sar1}\lim_{N\to\infty}\frac1N\sum_{n\leq N}f(T^nx)\mob(n)=0\eeq
for all continuous $f : X \rightarrow \R$ and all $x \in X$; here $\mu$ stands for the M\"obius function.

Let $M(X, T)$ denote the set of Borel $T$-invariant probability measures on $X$; this is a simplex whose set of extremal points coincides with $M^{e}(X, T)$, the set of ergodic measures.
Any measure $\nu$ in $M(X, T)$ gives rise to a measure-theoretic dynamical system $(X, \mathcal{B}, \nu, T)$ with $\mathcal{B}=\mathcal{B}(X)$ the Borel $\sigma$-algebra of subsets of $X$. Such a system is called  {\em rigid} if there exists a strictly increasing sequence $\{q_n\}_{n \geq 1}$ such that for every $g\in L^2(X,\nu)$, $g\circ T^{q_n}\to g$ in $L^2(X,\nu)$ as $n \rightarrow \infty$. Rigid systems have (measure-theoretic) entropy zero, so by the variational principle,  one naturally expects the validity of~\eqref{sar1} in the topological systems whose all invariant measures yield rigidity.

Considerable progress in understanding Sarnak's conjecture has been accomplished following the result of Matom\"aki and the third author  \cite{Ma-Ra} which roughly says that the behavior of $\mu$ on typical short intervals is the same as its behavior on intervals $[1,N]$. In trying to prove~\eqref{sar1} for rigid systems we need a variant of \cite{Ma-Ra} that holds for almost all short arithmetic progressions. Such an extension is obtained in Theorem~\ref{Thm:ad} (see also \cite{Mangerel} for a very recent related result). Somewhat surprisingly we encounter a significant technical complication that requires us to assume that $\sum_{p|q}1/p$ is not too large compared to $\sum_{p \leq H} 1/p$. Roughly speaking this means that we need to avoid those $q$ that have a lot of prime factors $\leq H$.
We call $\sum_{p | q} 1/p$ the {\em prime volume} of $q$. The prime volume grows slowly with $q$: in Lemma~\ref{lem:qa} and Corollary~\ref{cor:densone}, we observe that
\begin{equation}\label{ogrmax}
\sum_{p|q}\frac1p\leq \log\log\log q+O(1)
\end{equation}
but ``most'' of the time, the prime volume of $q$ stays bounded. Precisely, if we set
\begin{equation}\label{eq:cd}
\mathcal{D}_j:=\Big\{q\in \N\;:\; \sum_{p|q}\frac{1}{p}<j\Big\},
\end{equation}
where obviously $\mathcal{D}_j\subset \mathcal{D}_{j+1}$ for $j\in \N$, then
\begin{equation}\label{ogrmin}
  d(\mathcal{D}_j)\to1\text{ when }j\to\infty.\end{equation}
Note that the density $d(\mathcal{D}_j)$ of $\mathcal{D}_{j}$ exists for each $j$ and equals $\mathbb{P}(\sum_{p}X_p<j)$ with $\{X_{p}\}$ a sequence of Bernoulli independent random variables with $\mathbb{P}(X_{p} = 1) = 1 - \mathbb{P}(X_p = 0) = 1/p$.

Before we state our main results, we need to define the class of dynamical systems for which our results will be applicable.

\begin{Def}\label{def:good}\em Let $(X,T)$ be a topological dynamical system.  We say that $(X,T)$ is {\em good} if for every $\nu\in M(X,T)$ at least one of the following conditions holds:\\
(BPV rigidity): $(X, \mathcal{B}, \nu, T)$ is rigid along a sequence $\{q_n\}_{n\geq1}$ with {\bf bounded prime volume}, i.e.\ there exists $j$ such that $\{q_n\}_{n\geq1}\subset \mathcal{D}_j$;\\
(PR rigidity): $(X,\mathcal{B},\nu,T)$ has {\bf polynomial rate} of rigidity: there exists a  linearly dense (in $C(X)$) set $\mathcal{F}\subset C(X)$ such that for each $f\in\mathcal{F}$ we can find $\delta>0$ and a sequence $\{q_n\}_{n\geq1}$ satisfying
\begin{equation}\label{eq:conv0}
\sum_{j=-q_n^\delta}^{q_n^\delta}\|f\circ T^{jq_n}-f\|_{L^2(\nu)}^2\to 0.
\end{equation}
\end{Def}

Here is our first result.

\begin{Th}\label{thm:main}
Let $(X,T)$ be a {\bf good} topological dynamical system.  Then $(X,T)$ is M\"obius disjoint, that is,~\eqref{sar1} holds for all continuous $f : X \to \R$ and all $x \in X$.
\end{Th}

To put our second result in context, we recall also a recent result of Frantzikinakis-Host \cite{Fr-Ho} according to which Sarnak's conjecture holds in the {\bf logarithmic form} for all zero entropy topological dynamical systems with at most countably many ergodic invariant measures.

\begin{Th}\label{thm:mainE}
Let $(X,T)$ be a topological dynamical system. Suppose that the set of ergodic measures $M^e(X,T)$ is countable. If for every $\nu\in M^e(X,T)$ the measure-theoretic dynamical system $(X, \mathcal{B}, \nu, T)$ satisfies either BPV rigidity or PR rigidity (see Definition~\ref{def:good}), then $(X,T)$ is M\"obius disjoint; that is,~\eqref{sar1} holds for all continuous $f : X \rightarrow \R$ and all $x \in X$.
\end{Th}

One of the main techniques for establishing Sarnak's conjecture is the Daboussi-Delange-K\'atai-Bourgain-Sarnak-Ziegler criterion (colloquially

known as DDKBSZ, see \cite{Bo-Sa-Zi}, \cite{Kat}) and its measure-theoretic counterpart called AOP \cite{Ab-Ku-Le-Ru}. This criterion reduces establishing M\"obius disjointness to understanding the \textit{joinings} of $T^{p}$ and $T^{q}$ for distinct sufficiently large primes $p$ and $q$. For rigid systems the set of joinings is typically very complicated, making it impossible to appeal to the DDKBSZ criterion in this setting. In fact, only two general criteria that go beyond the DDKBSZ criterion are known:
  \begin{enumerate}
  \item \label{it:first} M\"obius disjointness holds for dynamical systems whose all invariant measures yield systems with discrete spectrum \cite{Hu-Wa-Zh} (see also Cor.\ 3.20 in \cite{Fe-Ku-Le}).
  \item \label{it:second} M\"obius disjointness holds for systems with all invariant measures of subpolynomial complexity \cite{Hu-Wa-Ye}.
  \end{enumerate}
  Note that the dynamical systems in \eqref{it:first} are rigid, since systems with discrete spectrum have a countable group of eigenvalues. In fact, all the examples given in \cite{Hu-Wa-Ye} for which the second criterion \eqref{it:second} applies are also rigid.
  Hence, our result can be seen as a new one on the short list of those which are complementary to the DDKBSZ criterion.

  In  Section~\ref{s:jakstosowac}, we give  a short outline of  the differences between BPV rigidity and PR rigidity, and discuss their common features. 
  Both conditions introduce a kind of stability in our M\"obius disjointness results: if in one uniquely ergodic model of a measure-preserving automorphism, BPV or PR rigidity holds then in all remaining uniquely ergodic models of the automorphism we have M\"obius disjointness (see Corollary~\ref{c:momo1} and Theorem~\ref{t:main2}). This applies, in particular, to all examples which are considered in Corollaries~\ref{c:rig1a},~\ref{c:rig3},~\ref{cor:spflow} and~\ref{rokext}.

Let us now highlight the old and new results that follow as special cases of Theorem~\ref{thm:main} and Theorem~\ref{thm:mainE}. In Subsection \ref{sub:iet}, we prove the following corollary:

\begin{Cor}\label{c:rig1a} For every $d\geq 2$, almost every  interval exchange transformation (IET) of $d$ intervals is M\"obius disjoint.
\end{Cor}
Let us recall that recently Chaika and Eskin \cite{Ch-Es} proved Corollary~\ref{c:rig1a} {\bf for $d=3$} (see also \cite{Bo}, \cite{Fe-Ma}, \cite{Kar} for M\"obius disjointness for some subclasses of interval exchange transformations with~3 intervals).

By $\T$ we denote the additive circle which we identify with $[0,1)$ with addition mod~1, and
we consider now {\em Anzai skew products} on $\T^2$, that is, systems of the form
$$
T_\phi(x,y):=(x+\alpha,y+\phi(x)),
$$
for irrational $\alpha\in \T$ and $\phi:\T\to\T$ a continuous function on the circle.
If $\phi$ is of zero topological degree, that is, $\phi(t)=\widetilde{\phi}(t)~{\rm mod~1}$, where $\widetilde{\phi}:\R\to\R$ is absolutely continuous and 1-periodic, then
$T_\phi^{q_n} \rightarrow \text{Id}$ {\bf uniformly} along the sequence $\{q_{n}\}$ of best rational approximations to $\alpha$ \cite{He}.
In \cite{Wa}, Wang proved that all such analytic Anzai skew products are M\"obius disjoint (for an earlier result, see Liu-Sarnak \cite{Li-Sa}), see also~\cite{Hu-Wa-Ye} for an extension of this result to the $C^\infty$ case.
 In Section \ref{sub:anz} we will also prove the following corollary:

\begin{Cor}\label{c:rig3} For every $\epsilon>0$, each irrational $\alpha$ and $\phi$ of zero topological degree, $\int_0^1\widetilde{\phi}\,dt=0$ and of class $C^{2+\epsilon}$, the corresponding Anzai skew product $T_\phi$ is M\"obius disjoint.\end{Cor}

Another important class of systems with the property that $T^{q_{n}} \rightarrow \text{Id}$ (uniformly) along some subsequence $\{q_n\}$ is given by smooth (area preserving) flows on $\T^2$. In fact, whenever such a flow has no fixed points, it is rigid.
In particular, the recent result of \cite{Hu-Xu} that such $C^\infty$ flows are M\"obius disjoint is a particular case of our main result whenever $\{q_n\}$ is of bounded prime volume. If we drop the assumption of the boundedness of the prime volume of $\{q_n\}$ then still M\"obius disjointness will hold whenever the roof function is smooth enough, see Corollary~\ref{cor:spflow}. Finally, in Corollary~\ref{rokext}, we prove new M\"obius disjointness results for so called Rokhlin extensions.

Furthermore, Theorem \ref{th:main1}, which also holds for the Liouville function $\lio$ instead of $\mob$, has some purely arithmetic consequences for the correlations of the Liouville function. Matom\"aki and the third author showed in \cite{Ma-Ra} that
\beq\label{matrad1}
\limsup_{N\to\infty}\left|\frac1N\sum_{n\leq N}\lio(n)\lio (n+h)\right|<1-\eta(h)\eeq
for some $\eta(h)>0$ and all $h\geq1$. We can weaken the dependence of $\eta$ on $h$, along a subsequence of any given subsequence of $N$.

\begin{Cor}\label{c:lio1} For each sequence  $\{N_n\}_{n \geq 1}\subset\N$ there exists a subsequence $\{N_{n_k}\}$ such that for each $j\geq1$ and some $\eta>0$ (depending on the subsequence and~$j$), we have:
$$
\lim_{k\to\infty}\frac1{N_{n_k}}
\sum_{n\leq N_{n_k}}\lio(n)\lio(n+h)
<1-\eta,$$
for all $h\in\mathcal{D}_j$.
 \end{Cor}
The proof of the above corollary uses the notion of {\em Furstenberg systems} associated to $\lio$ and will be given in Section~\ref{sec:mulf}.

We remark that Corollary \ref{c:lio1} is related to a recent result of Tao and Ter\"av\"ainen \cite{T-T} according to which Chowla's conjecture holds for ``most'' subsequences. However, given a particular subsequence $\{N_{n_k}\}$ their result cannot guarantee the existence of even one subsequence of $\{N_{n_k}\}$ along which Chowla would hold. This is not surprising, since such a result would settle Chowla's conjecture. This marks an important but subtle distinction between their result and our Corollary \ref{c:lio1}.

\subsection*{Acknowledgments}

MR acknowledges the partial support of a Sloan fellowship and NSF grant DMS-1902063. ML's research is supported by Narodowe Centrum Nauki grant 2019/33/B/ST1/00364.
We would like to thank the American Institute of Mathematics for hosting a workshop on ``Sarnak's Conjecture'' at which this work was begun. We are grateful to Sacha Mangerel and Joni Ter\"av\"ainen for bringing to our attention an issue in the proof of Theorem~\ref{Thm:ad} in the previous manuscript and to Krzysztof Fr\c{a}czek for discussions on Theorem~\ref{t:main2}.

\section{Completely rigid points in topological dynamics}
To prove Theorem \ref{thm:main} and Theorem \ref{thm:mainE} we will need the notion of completely rigid points, or, taking into account Definition~\ref{def:good}, nicely rigid points.
\begin{Def}\em
  Given a topological dynamical system $(X, T)$ a point $x \in X$ is {\em quasi-generic with respect to a measure $\nu$} if there exists a subsequence $M_k \rightarrow \infty$ along which we have,
  \begin{equation} \label{eq:convergence}
  \frac{1}{M_k} \sum_{m \leq M_k} \delta_{T^m x} \rightarrow \nu
  \end{equation}
  weakly in the space of probability measures on $X$.
\end{Def}
 Notice that by \eqref{eq:convergence}, necessarily,  $\nu\in M(X,T)$.

We are now ready to define {\em completely rigid (nicely rigid) points} in analogy with the {\em completely deterministic points} introduced in \cite{Ka}, \cite{We}.

\begin{Def}\em Given a topological dynamical system $(X, T)$, a point $x \in X$ is {\em completely rigid}  if each measure $\nu \in M(X,T)$ for which $x$ is quasi-generic yields a rigid measure-theoretic dynamical system. If, additionally, each such $\nu$ satisfies either BPV  or PR rigidity condition then $x$ is called {\em nicely rigid}.\end{Def}

Alternatively, we can restate the definition of $x$ being nicely rigid by requiring  that whenever
$$\frac1{M_k}\sum_{m\leq M_k}\delta_{T^mx}\to\nu,$$
then $(X,\mathcal{B}(X),\nu,T)$ satisfies either BPV rigidity or PR rigidity condition. In Section~\ref{s:main1}, we will prove the following more general version of Theorem~\ref{thm:main}.

\begin{Th} \label{th:main1} Assume that $(X,T)$ is a topological dynamical system. Assume that $x\in X$ is nicely rigid. Then
\beq\label{teza}\lim_{N\to\infty}\frac1N\sum_{n\leq N}f(T^nx)\mob(n)=0\eeq
for each $f\in C(X)$.\end{Th}

It follows that whenever in $(X,T)$ all points  are nicely rigid, the system $(X, T)$  satisfies Sarnak's conjecture (when PR rigidity condition holds, we obtain~\eqref{teza} for $f\in \mathcal{F}$ (and all $x\in X$), but it is not hard to see that~\eqref{teza} is a closed condition in $C(X)$).
Theorem \ref{thm:main} is now an immediate consequence of Theorem~\ref{th:main1}.


\subsection{Consequences for multiplicative functions}\label{sec:mulf}
\begin{Remark}\label{r:local}\em
  A closer look at the proof of Theorem~\ref{th:main1} shows that it has a ``local'' character. Namely, to show~\eqref{teza} for {\bf fixed} $x\in X$ and $f\in C(X)$ we need to know that for each measure $\nu$ for which $x$ is quasi-generic we have, for some $\{q_n\}$, $f\circ T^{q_n}\to f$ in $L^2(X,\nu)$ and either $\{q_n\}\subset\mathcal{D}_j$ (for some $j\geq1$) or (for some $\delta>0$)~\eqref{eq:conv0} holds.
  We will say that in this situation $(x,f)$ satisfies either BPV or PR rigidity, respectively.

Note also that if $\frac1{N_k}\sum_{n\leq N_k}\delta_{T^nx}\to\nu$, and we are interested in establishing \eqref{teza} along $\{N_k\}$, then we need $\{q_n\}$ (depending now on $(f,x)$ and $\{N_k\}$) to satisfy the above conditions. In such a situation we will say that $(x, f, \nu)$ satisfies either BPV or PR rigidity, depending on whether $\{q_{n}\} \subset \mathcal{D}_{j}$ or whether \eqref{eq:conv0} holds.
\end{Remark}

By treating $\mob$ as a point in the space $\{-1,0,1\}^{\Z}$ ($\mob(-n)=\mob(n))$ on which we have the action of the left shift $S$, we obtain the M\"obius subshift $(X_{\mob},S)$, where $X_{\mob}:=\overline{\{S^n\mob:\:n\in\Z\}}$.
On $X_{\mob}$ we consider all measures $\kappa$ for which $\mob $ is quasi-generic (i.e.\ $\kappa=\lim_{k\to\infty}\frac1{N_k}\sum_{n\leq N_k}\delta_{S^n\mob}$). Each measure-theoretic system $(X_{\mob},\mathcal{B}(X_{\mob}), \kappa,S)$ is called a {\em Furstenberg system} of $\mob$.

Consider now the M\"obius subshift  $(X_{\mob},S)$ with the function $\theta:X_{\mob}\to \{0,\pm1\}$, $\theta(y)=y(0)$ for $y\in X_\mu$. Since the set of square-free numbers has positive density, $\mob$ is not orthogonal to itself (so~\eqref{teza} is not satisfied for $\theta\in C(X_{\mob})$ and $\mu\in X_{\mob}$ {\bf along any subsequence of} $\N$), Theorem~\ref{th:main1} implies:

\begin{Cor}\label{c:mob1} For no $\kappa\in M(X_\mu,S)$ for which $\mu$ is quasi-generic, $(\mu,\theta,\kappa)$ satisfy either BPV or PR rigidity. In particular, no Furstenberg system $(X_{\mob},\mathcal{B}(X_{\mob}),\kappa,S)$ of $\mu$ is BPV or PR rigid. The same holds for the Liouville function $\lio$.\end{Cor}

Let us continue with $\lio$.

Assume that we are given any Furstenberg system $(X_{\lio},\mathcal{B}(X_{\lio}),\kappa,S)$ of $\lio$, that is, assume that $\frac1{N_k}\sum_{n\leq N_k}\delta_{S^n\lio}\to\kappa$. Then
$$
\lim_{k\to\infty}\frac1{N_k}\sum_{n\leq N_k}\lio(n)\lio(n+h)=
\lim_{k\to\infty}\frac1{N_k}\sum_{n\leq N_k}(\theta\cdot \theta\circ S^h)(S^n\lio)=$$$$\int \theta\cdot\theta\circ S^h\,d\kappa=\widehat\sigma_\theta[h]$$
for each $h\geq1$  ($\widehat\sigma_\theta(h)$ stands for the $h$th Fourier coefficient of the spectral measure $\sigma_{\theta}$ of $\theta\in L^2(X_{\lio},\kappa)$). Note that $\int_{X_{\lio}}\theta^2\,d\kappa=1$, so if   $\widehat{\sigma}_\theta[q_n]\to1$ along a subsequence $\{q_n\}$,  then $\theta$ would be rigid.
Since Corollary~\ref{c:mob1} holds, we can now prove Corollary~\ref{c:lio1}:
\begin{proof}[Proof of Corollary \ref{c:lio1}]
By the above, we have that there exists $\eta_0>0$ such that
$$
\limsup_{\mathcal{D}_j\ni h\to\infty}\lim_{k\to\infty}\frac1{N_k}\left|\sum_{n\leq N_k}\lio(n)\lio(n+h)\right|<1-\eta_0.$$
To complete the proof, we use \eqref{matrad1} (for ``small'' $h$).
\end{proof}


\section{Properties of M\"obius function and the proof of Theorem~\ref{th:main1}}
\subsection{M\"obius function on short intervals along arithmetic progressions}\label{Newproperty}

We will prove the following result.
\begin{Th}\label{Thm:ad}
  Let $A > 100$ and $\varepsilon \in (0, \tfrac{1}{100})$ be given.
  For all $X > X_0(\varepsilon, A)$, $H > H_0(\varepsilon, A)$ and $q \leq (\log X)^{A}$ such that
  \begin{equation} \label{eq:cond}
  \sum_{p | q} \frac{1}{p} \leq (1 - \varepsilon) \sum_{p \leq H} \frac{1}{p},
  \end{equation}
  we have
  $$
  \frac{1}{q X H} \sum_{a < q} \int_{X}^{2X} \Big | \sum_{\substack{x \leq n \leq x + q H \\ n \equiv a \pmod{q}}} \mu(n) \Big | dx \leq \varepsilon.
  $$
\end{Th}

One can roughly think of the condition on $q$ as excluding those $q$ that are divisible by primorials such as $\prod_{p \leq H} p$. It is a generic condition as we show in the lemma below.

\begin{Lemma}\label{lem:qa}
  Let $\varepsilon \in (0, \tfrac{1}{100})$ be given. For any $A > 100$,
  \begin{equation*}
\# \Big \{ q \leq X :   \sum_{p | q} \frac{1}{p} \geq ( 1 - \varepsilon) \sum_{p \leq H} \frac{1}{p} \Big \} \ll_{A} X (\log H)^{-A}.
  \end{equation*}
  Furthermore, for all $q$ sufficiently large,
  $$
  \sum_{p | q} \frac{1}{p} \leq \log\log\log q + O(1).
  $$
  Therefore, \eqref{eq:cond} holds for all $H > \exp((\log \log q)^{1 + 2 \varepsilon})$ and $q$ sufficiently large in terms of $\varepsilon^{-1}$.
\end{Lemma}

\begin{proof}
  By Chernoff's bound the number of exceptions is bounded by
  $$
\ll  (\log H)^{-2 A ( 1 - \varepsilon) } \sum_{q \leq X} \exp \Big ( 2 A \sum_{p | q} \frac{1}{p} \Big )  .
  $$
  By \cite[Theorem 5, Ch. III.3]{Te}, the above is
  $$
  \ll (\log H)^{-2 A (1 - \varepsilon)} X \prod_{p \leq X} \Big ( 1 + \frac{e^{2 A / p} - 1}{p} \Big ) \ll_{A} X (\log H)^{-A}.
  $$
  For the second claim, we notice that $q$ has at most $2 \log q$ distinct prime factors, since $\prod_{p < 2 \log q} p > q$ for all $q$ sufficiently large. Therefore,
  $$
  \sum_{p | q} \frac{1}{p} \leq \sum_{p \leq 2 \log q} \frac{1}{p} \leq \log \log \log q + O(1)
  $$
  as claimed.
\end{proof}
The upper bound $\log\log\log q$ is attained for example by primorials $q = \prod_{p \leq K} p$ once $K$ is sufficiently large. The first part of the above lemma implies the following:
\begin{Cor}\label{cor:densone} For $j\in \N$ let $\mathcal{D}_j$ be as in \eqref{eq:cd}. Then
$
\lim_{j\to+\infty} d(\mathcal{D}_j)=1.
$
\end{Cor}
\begin{proof} Let $A=1$, $\varepsilon=1/200$ and let $H=\exp(j)$. Then
$$
D_j^c\subset \Big \{ q \in \N :   \sum_{p | q} \frac{1}{p} \geq ( 1 - \varepsilon) \sum_{p \leq H} \frac{1}{p} \Big \}.
$$
The statement then immediately follows by Lemma \ref{lem:qa}.
\end{proof}

Before we prove the above theorem, let us state the following immediate corollary.

\begin{Cor}\label{cor1} For each $\vep \in (0, \tfrac{1}{100})$ there exists $L_0$ such that for each $L\geq L_0$ and $q \geq 1$ with
  $$
  \sum_{p | q} \frac{1}{p} \leq (1 - \varepsilon) \sum_{p \leq L} \frac{1}{p},
  $$
  we can find $M_0=M_0(q,L)$ such that, for all $M\geq M_0$, we have
\beq\label{eq:klucz}
\sum_{j=0}^{M/Lq}\sum_{a=0}^{q-1}\left|\sum_{m\in [z+ jLq, z+ (j+1)Lq)\atop m\equiv a\text{ mod }q}\mob(m)\right| \leq \vep M
\eeq
for some $0\leq z<Lq$.
\end{Cor}
\begin{proof}
  Partitioning into dy-adic intervals we readily see that the conclusion of Theorem \ref{Thm:ad} might have been as-well stated with an integration over $0 \leq x \leq X$ instead of an integration over $X \leq x \leq 2X$. It follows from this that for any $\varepsilon > 0$ and $q \geq 1$ and all $L$ and $M$ sufficiently large in terms of $\varepsilon$ and $q$,
  $$
 \sum_{0 \leq z < q L} \Big ( \sum_{\substack{0 < j < M / (L q)}} \sum_{a < q} \Big | \sum_{\substack{n \in [z + j q L, z + (j + 1) q L] \\ n \equiv a \mod{q}}} \mu(n) \Big | \Big ) \leq \varepsilon  M \cdot L q.
  $$
 The claim now follows by the pigeon-hole principle.
\end{proof}

\begin{proof}[Proof of Theorem \ref{Thm:ad}]
  Grouping terms according to $d = (a,q)$, we have the bound
  $$
  \frac{1}{q X H} \sum_{a < q} \int_{X}^{2X} \Big | \sum_{\substack{x \leq n \leq x + q H \\ n \equiv a \pmod{q}}} \mu(n) \Big | dx \leq
  $$
  $$\frac{1}{q X H} \sum_{d | q} \sum_{\substack{0 \leq a < q / d \\ (a,q/d) = 1}} \int_{X}^{2X} \Big | \sum_{\substack{x / d \leq n \leq x / d + (q / d) H \\ n \equiv a \pmod{q / d}}} \mu(d n) \Big | dx.
  $$
  Let $\varepsilon > 0$.
  We now introduce the set $\mathcal{S}_{d}$ which consists of integers $n$ that have at least one prime divisor in each of the intervals $[P_i, Q_i]$ where $Q_1 = d H$, $P_1 = \log^{2000} Q_1$ and $P_j, Q_j$ are as in \cite[equation (4)]{Ma-Ra}, that is,
  $$
  P_j = \exp(j^{4j} (\log Q_1)^{j - 1} \cdot \log P_1) \ , \ Q_j = \exp(j^{4j + 2} (\log Q_1)^j) \ , \ j \leq J
  $$
  with $J$ the largest index such that $Q_J \leq \exp(\sqrt{\log X})$. We then find that the contribution of the integers not in $\mathcal{S}_d$ is
  $$
  \leq \frac{1}{q X H} \sum_{d | q} \sum_{\substack{0 \leq a < q / d \\ (a, q / d) = 1}} \int_{X}^{2X} \Big ( \sum_{\substack{x / d \leq n \leq x / d + (q/d) H \\ n \equiv a \pmod{q / d} \\ n \not \in \mathcal{S}_{d}}} 1 \Big ) dx.
  $$

  By Fubini's theorem, this is
  \begin{equation} \label{eq:sieveout}
  \leq \frac{1}{X} \sum_{d | q} \sum_{\substack{X / d \leq n \leq 2 X / d \\ (n, q / d) = 1 \\ n \not \in \mathcal{S}_{d}}} 1 .
\end{equation}
Notice that, for any $j \geq 1$,
\begin{equation}\label{eq:saa}
\sum_{\substack{P_j \leq p \leq Q_j \\ p \nmid q}} \frac{1}{p} \geq \sum_{P_j \leq p \leq Q_j} \frac{1}{p} - \sum_{p | q} \frac{1}{p} \geq \frac{\varepsilon}{2} \log\log H + 2 \log j.
\end{equation}
Therefore, by a standard sieve bound,
  $$
  \sum_{\substack{X / d \leq n \leq 2X / d \\ (n, q / d) = 1 \\ n \not \in \mathcal{S}_{d}}} 1 \ll \frac{X}{d} \frac{\varphi(q / d)}{q / d} \sum_{j = 1}^{J} \prod_{\substack{P_j \leq p \leq Q_j \\ p \nmid q}} \Big (1 - \frac{1}{p} \Big ) \ll (\log H)^{-\varepsilon / 2} \frac{\varphi(q / d)}{q} \cdot X \leq $$
  $$\varepsilon \cdot \frac{\varphi(q / d)}{q} \cdot X
  $$
  for all sufficiently large $H$,
  and therefore,~\eqref{eq:sieveout} is $\ll \varepsilon$. Thus, it suffices to bound
  \begin{align*}
    \frac{1}{ q X H} & \sum_{d | q} \sum_{\substack{0 \leq a < q / d \\ (a , q / d) = 1}} \int_{X}^{2X} \Big | \sum_{\substack{x / d \leq n \leq x / d + (q / d) H \\ n \equiv a \pmod{q / d} \\ n \in \mathcal{S}_{d}}} \mu(d n) \Big | dx \\ & \leq \Big ( \frac{1}{q X H} \sum_{d | q} \sum_{\substack{0 \leq a < q / d \\ (a, q / d) = 1}} X \Big )^{1/2} \cdot\\& \Big ( \frac{1}{q X H} \sum_{d | q} \sum_{\substack{0 \leq a < q / d \\ (a , q / d) = 1}} \int_{X}^{2X} \Big | \sum_{\substack{x / d \leq n \leq x / d + (q / d) H \\ n \equiv a \pmod{q / d} \\ n \in \mathcal{S}_{d}}} \mu(d n) \Big |^2 dx \Big )^{1/2}.
  \end{align*}
  Therefore, to conclude, it will suffice to show that
  $$
  \frac{1}{q X H^2} \sum_{d | q} \sum_{\substack{0 \leq a < q / d\\ (a,q/d) = 1}} \int_{X}^{2X} \Big | \sum_{\substack{x / d \leq n \leq x / d + (q / d) H \\ n \equiv a \pmod{q / d}\\ n \in \mathcal{S}_{d}}} \mu(d n) \Big |^2 dx \ll \varepsilon^2.
  $$
  We express the condition $n \equiv a \pmod{q / d}$ using Dirichlet characters. This allows us to re-write the above equation as
  $$
  \frac{1}{q X H} \sum_{d | q} d \cdot \frac{1}{\varphi(q / d)} \sum_{\chi \pmod{q / d}} \int_{X / d}^{2 X / d} \Big | \sum_{\substack{x \leq n \leq x + (q / d) H \\ n \in \mathcal{S}_{d}}} \mu(d n) \chi(n) \Big |^2 dx.
  $$
  We now claim that a variant of \cite[Lemma 14]{Ma-Ra} gives
  \begin{align} \nonumber
  \frac{1}{X / d} & \frac{1}{((q / d) H)^2}\int_{X / d}^{2X / d} \Big | \sum_{\substack{x \leq n \leq x + (q / d) H \\ n \in \mathcal{S}_{d}}} \mu(d n) \chi(n) \Big |^2 dx  \label{eq:firstintegral}  \ll \frac{1}{(\log X)^{100 A}} + \\& +\int_{(\log X)^{100 A}}^{X / (q H)} |D_{d}(1 + it, \chi)|^2 dt + \max_{T > X / (q H)} \frac{X / (q H)}{T} \int_{T}^{2T} |D_{d}(1 + it, \chi)|^2 dt,
  \end{align}
  where
  $$
  D_{d}(1 + it, \chi) := \sum_{\substack{X \leq n \leq 4 X \\ n \in \mathcal{S}_{d}}} \frac{\mu(d n) \chi(n)}{n^{1 + it}}.
  $$
  To obtain this we repeat the proof in  \cite[Lemma 14]{Ma-Ra} but choose $T_0 = (\log X)^{100 A}$, $h_2 = X / (\log X)^{1000 A}$ and notice that the prime number theorem in arithmetic progressions gives
  $$
  \frac{1}{h_2} S_{2}(x) \ll (\log X)^{-10000 A}.
  $$

  We will focus only on bounding the first integral in \eqref{eq:firstintegral} since the contribution of the second integral is handled by simply repeating the argument.
  Therefore, it remains to obtain a bound of the form $\ll \varepsilon^2$ for
  $$
  \frac{1}{q H^2} \sum_{d | q} \frac{1}{\varphi(q / d)} \sum_{\chi \pmod{q / d}} \Big ( \frac{q H}{d} \Big )^2 \int_{(\log X)^{100 A}}^{X / (q H)} |D_{d}(1 + it, \chi)|^2 d.t
  $$
  We re-write this as
  \begin{equation} \label{eq:tobound}
  \sum_{d | q} \frac{q / d}{\varphi(q / d)} \cdot \frac{1}{d} \sum_{\chi \pmod{q / d}} \int_{(\log X)^{100 A}}^{X / (q H)} |D_{d}(1 + it, \chi)|^2 dt.
  \end{equation}

  We claim that a slight modification of \cite[Proposition 1]{Ma-Ra} gives the following ``hybrid variant'' of \cite[Proposition 1]{Ma-Ra},
  \begin{equation} \label{eq:hybrid}
  \sum_{\chi \pmod{q / d}} \int_{(\log X)^{100 A}}^{X / (q H)} |D_{d}(1 + it, \chi)|^2 dt \ll_{\eta} \Big ( \frac{(\log Q_1)^{1/3}}{P_1^{1/6 - \eta}} + \frac{1}{(\log X)^{50 A}} \Big ) \cdot \frac{\varphi(q / d)}{q / d}
  \end{equation}
  for any $\eta > 0$.
  Inserting this ``hybrid variant'' into \eqref{eq:tobound} shows that \eqref{eq:tobound} is
  $$
  \ll \sum_{d | q} \frac{1}{d} \cdot \Big ( (\log d H)^{-2} + \frac{1}{(\log X)^{10 A}} \Big ) \ll_{\varepsilon} \frac{1}{\log H} + \frac{1}{\log^{A} X}
  $$
  (since $\sum_{d | q} d^{-1} \ll \log H$ by the assumption $\sum_{p | q} p^{-1} \leq \log\log H$)
  which is less than $\varepsilon$ for sufficiently large $H$ and $X$.

  Therefore, all that remains to be done is to explain how to obtain \eqref{eq:hybrid}.
  We repeat the argument in \cite[Proposition 1]{Ma-Ra} with the following differences:
  \begin{itemize}
    \item Instead of $Q_{v, H_j}(s)$ we work with
  $$
  Q_{v, H_j}(s, \chi) := \sum_{\substack{P_j \leq q \leq Q_j \\ e^{v / H_j} \leq q \leq e^{(v + 1) / H_j}}} \frac{\chi(q)}{q^s}
  $$
  and, instead of $R_{v, H_j}(s)$, we work with
  $$
  R_{v, H_j}(s, \chi) := \sum_{\substack{X e^{-v / H_j} \leq m \leq 2 X e^{-v / H_j} \\ m \in \mathcal{S}_{j}}} \frac{\mu(d m)\chi(m)}{m^s} \cdot \frac{1}{\# \{ P_j \leq p \leq Q_j : p | m\} + 1},
  $$
  where $\mathcal{S}_{j}$ is the set of those integers which have at least one prime factor in every interval $[P_i, Q_i]$ with $i \neq j$ and $i \leq J$.
\item We define $\mathcal{T}_j$ as the set of those $(t, \chi)$ with $(\log X)^{A} \leq t \leq X / (q H)$ and $\chi$ of modulus $q / d$ for which $|Q_{v,H_j}(1 + it, \chi)| < e^{-\alpha_j v / H_j}$ for all $v$, and $\mathcal{U}$ as the set of those $(t, \chi)$ that do not belong to any of the $\mathcal{T}_1 \cup \ldots \cup \mathcal{T}_{j}$.

\item In Section 8.1, instead of using the standard mean-value theorem, we use the following ``hybrid mean-value theorem'' (see \cite[Theorem 6.4]{Mo} for a proof):
  \begin{equation} \label{eq:hybrid2}
  \sum_{\chi \pmod{q / d}} \int_{0}^{T} \Big | \sum_{X \leq n \leq 4X} \frac{a(n)}{n^{1 + it}} \Big |^2 \ll \Big ( \frac{X}{q H} \cdot \varphi(q / d) + X \Big ) \sum_{\substack{X \leq n \leq 4X \\ (n, q/ d) = 1}} \frac{|a(n)|^2}{n^2}.
\end{equation}
This leads to a bound for $E_1$ which is
$$
E_{1} \ll \frac{(\log Q_1)^{1/3}}{P_1^{1/6 - \eta}} \cdot \Big ( \varphi \Big ( \frac{q}{d} \Big ) \frac{Q_1}{q H} + 1 \Big ) \cdot \frac{\varphi(q / d)}{q / d} \ll \frac{(\log Q_1)^{1/3}}{P_1^{1/6 - \eta}} \cdot \frac{\varphi(q / d)}{q / d}
$$
and explains the choice of $Q_1$.
\item In Section 8.2 we again appeal to the ``hybrid mean-value theorem'' instead of the standard mean-value theorem. In particular, repeating the proof of Lemma 13 with the hybrid mean-value theorem, we obtain
  \begin{align*}
  \sum_{\chi \pmod{q / d}} & \int_{(\log X)^{A}}^{X / (q H)} |Q_{r, H_{j - 1}}(1 + it, \chi)^{\ell_j, r} \cdot  R_{v, H_j}(1 + it, \chi)|^2 dt \\ & \ll \Big ( \frac{\varphi(q / d)}{q H} + Q_{j - 1} \Big ) \exp(2 \ell_{j,r} \log \ell_{j, r}) \cdot \frac{\varphi(q / d)}{q / d}.
  \end{align*}
  This leads to the bound $$E_{j} \ll \frac{1}{j^2} \cdot \frac{1}{P_1} \cdot \frac{\varphi(q / d)}{q / d}.$$
\item In Section 8.3 it suffices to bound
    $$
    \sum_{(t, \chi) \in \mathcal{T}}  |Q_{v, H}(1 + it, \chi) R_{v, H}(1 + it, \chi)|^2
    $$
    with the sum taken over a $1$-well spaced set of tuples $(t, \chi)$ (that is if $(t, \chi) \neq (t', \chi)$ then $|t - t'| \geq 1$) with $\mathcal{T} \subset \mathcal{U}$. By the prime number theorem, we have $Q_{v, H}(1 + it, \chi) \ll (\log X)^{- 50A}$ for any $(t, \chi) \in \mathcal{T}$. Therefore, it remains to bound
    $$
    (\log X)^{-100 A} \sum_{(t, \chi) \in \mathcal{T}} |R_{v, H}(1 + it, \chi)|^2 .
    $$
    By a hybrid version of Halasz's Lemma (see \cite[Theorem 8.3]{Mo}), this is
    $$
    \ll (\log X)^{-100 A} \Big ( X e^{-v / H} + |\mathcal{T}| \cdot \Big ( \frac{q}{d} \cdot \frac{X}{q H} \Big )^{1/2} \Big ) \frac{\log X}{X e^{-v / H}}.
    $$
    Moreover, repeating the proof of \cite[Lemma 8]{Ma-Ra} using the hybrid mean-value \eqref{eq:hybrid2}, we obtain that $|\mathcal{T}| \ll (X / (d H))^{1/2 - \eta + o(1)}$. From this, it follows that the above expression is $\ll (\log X)^{-100 A}$ as needed.
  \end{itemize}
  Summing the error term that we obtain in the above modifications of Section 8.1, 8.2 and 8.3, we obtain the claimed bound \eqref{eq:hybrid}.
\end{proof}

\subsection{A short discussion of BPV and PR rigidity} \label{s:jakstosowac}
In this short section we would like to point out  differences and common features of BPV and PR rigidity, defined in Definition~\ref{def:good}.
BPV rigidity puts arithmetic restrictions on rigidity times but its advantage is that it depends only on the measure-theoretic properties of $\nu$, hence it does not ``see'' topological properties of $(X,T)$. On the other hand, PR rigidity does not impose any arithmetic restrictions on rigidity times, however it takes into account the topological properties of $(X,T)$ expressed by a speed of rigidity for $\nu\in M(X,T)$  and certain continuous functions (namely $f\in\mathcal{F}$).
If in a concrete situation we have a measure-theoretic system and BPV rigidity holds then M\"obius disjointness holds in {\bf all} uniquely ergodic models of the system (hence, in Corollary~\ref{c:rig1a}, we obtain that in {\bf each} uniquely ergodic model of an  IET from the set $C^c$ M\"obius disjointness holds). In contrast PR rigidity applies only in {\bf selected} uniquely ergodic models. However, whenever PR rigidity holds, we will  show in Corollary~\ref{c:momo1} that not only M\"obius disjointness holds but, in fact, we have the strong MOMO property in this model. This, by Theorem~\ref{t:main2}, implies that M\"obius disjointness holds in {\bf all} uniquely ergodic systems of the original measure-theoretic system (hence, also for uniquely ergodic systems considered in Corollary~\ref{c:rig3}, we  have M\"obius disjointness in all other uniquely ergodic models).

To illustrate this, consider an ergodic automorphism with discrete spectrum. By Halmos-von Neumann theorem, it has a uniquely ergodic model which is given as a rotation $Tx=x+a$ on a (metric) compact, abelian group $X$ considered  with Haar measure. Now, $\mathcal{F}$ is given by the group of characters of $X$ each of which is an eigenfunction. If $f\in\mathcal{F}$ corresponds to an eigenvalue $e^{2\pi i\alpha}$ with $\alpha\in[0,1)$  irrational then~\eqref{eq:conv0} is satisfied as it is $\ll \|q_n\alpha\|^2\sum_{j=-q_n^\delta}^{q_n^\delta}j^2\to0$  (for a small $\delta>0$), where $q_n$, $n\geq1$, stands for the sequence of denominators of $\alpha$. If $\alpha$ is rational then consider just multiples of its denominators. It follows that {\bf all} uniquely ergodic models of systems with discrete spectrum are M\"obius disjoint (cf.\ \cite{Hu-Wa-Zh} for the first proof of this result).  Note that BPV rigidity is satisfied for each irrational rotation $Tx=x+\alpha$ on $\T$  (this follows from Vinogradov's theorem on equidistribution of $p\alpha$, $p\in\PP$) but it is not clear how to apply it to any ergodic discrete spectrum automorphism.

  Finally, note that a natural situation in which PR rigidity holds is the case of {\em pointwise rigid} homeomorphisms satisfying the {\em topological PR rigidity} condition: for some $\{q_n\}$, $\delta>0$ and each $x\in X$, we have
\begin{equation}\label{topPR}
\sum_{j=-q_n^{\delta}}^{q_n^\delta}d(T^{jq_n}x,x)\to 0.\end{equation}
Indeed, assuming that the metric $d$ is bounded by~1, for $\mathcal{F}$ we can take the family of Lipschitz continuous functions and for each such function and each $\nu\in M(X,T)$, condition~\eqref{eq:conv0} follows automatically from~\eqref{topPR}. For different concepts of rigidity in topological dynamics, see \cite{Gl-Ma}.

\subsection{Interval exchange transformations}\label{sub:iet}
In this section we will prove Corollary \ref{c:rig1a}.  We first recall the following result of Chaika:
\begin{Th}[Corollary 3, \cite{Chaika}]\label{thm:ch} For every $d \in \N$ and $\epsilon>0$ there exists $0<a(\epsilon)<1$ such that for every $E\subset \N$ satisfying $\underline{d}(E)\geq a(\epsilon)$, there exists a set of $d$ IET's of measure at least $1-\epsilon$ such that each IET in this set has a rigidity sequence in $E$.
\end{Th}
Using the above result we are able to prove Corollary \ref{c:rig1a}:
\begin{proof}[Proof of Corollary \ref{c:rig1a}] Let $\epsilon_n=\frac{1}{n}$ and let $j_n$ be such that
$d(D_{j_n})\geq a(\epsilon_n)$ (such $j_n$ exists by Corollary \ref{cor:densone}). Then, by Theorem \ref{thm:ch},  all but a set $C_n$ of measure $\epsilon_n$ of  $d$ IET's have a rigidity sequence in $\mathcal{D}_{j_n}$ (we can WLOG assume that such IET's are also uniquely ergodic). Hence, BPV rigidity (see Definition~\ref{def:good}) is satisfied for every $d$ IET outside $C_n$. Hence, by Theorem~\ref{thm:main}, every such IET is M{\"o}bius disjoint. It is then enough to notice that the measure of $C=\bigcap_{n}C_n$ is $0$ and every IET outside $C$ is M{\"o}bius disjoint.
\end{proof}

\subsection{Anzai skew products}\label{sub:anz}
In this section we will prove Corollary \ref{c:rig3}. By an abuse of notation, we can treat $\phi$ as a zero mean $C^{2+\epsilon}$ real-valued function on $\T$: $\phi\in C^{2+\epsilon}(\T)$, and therefore
$$
\phi(x)=\sum_{j\in \Z}a_je_j(x)\;\;\text{  where }  \;\;|a_j|\ll j^{-2-\epsilon} \text{ and }\;\; a_0=0.
$$
Let $\{q_n\}$ denote the sequence of denominators of $\alpha$. We have the following lemma:
\begin{Lemma}\label{lem:srig} There exists a subsequence $\{q_{n_m}\}$ such that
$$
\sup_{(x,y)\in \T^2}\sup_{|k|\leq q_{n_m}^{\epsilon^2/4}}d(T_{\alpha,\phi}^{kq_{n_m}}(x,y),(x,y))\ll q_{n_m}^{-\epsilon^2/2}.
$$
\end{Lemma}

Before we prove the above lemma, let us show how it implies Corollary \ref{c:rig3}.
\begin{proof}[Proof of Corollary \ref{c:rig3}] Let $\{q_{n_m}\}$ denote the sequence from Lemma \ref{lem:srig}. We have
$$\sum_{|k|\leq q_{n_m}^{\epsilon^2/4}}d(T_{\alpha,\phi}^{kq_{n_m}}(x,y),(x,y))\ll q_{n_m}^{-\epsilon^2/4},$$ so condition~\eqref{topPR}  is satisfied and therefore PR rigidity is satisfied for the sequence $\{q_{n_m}\}$, $\delta=\epsilon^2/4$ and the family $\mathcal{F}$ of Lipschitz functions. The statement then follows by Theorem~\ref{thm:main}.
\end{proof}
So, it only remains to prove Lemma~\ref{lem:srig}.
\begin{proof}[Proof of Lemma \ref{lem:srig}]
Notice that $T_{\alpha,\phi}^{kq_n}(x,y)=(x+kq_n\alpha,y+S_{kq_n}(\phi)(x))$, where $S_r(\phi)(x):=\phi(x)+\phi(x+\alpha)+\ldots+\phi(x+(r-1)\alpha)$, for $r\geq 1$.

For every $n\in \N$, we have $S_{q_n}(\phi)(x)=\sum_{j\in \Z}a_j S_{q_n}(e_j)(x)$, where $e_j(x):=e^{2\pi i jx}$. Moreover,
$$
S_{q_n}(e_j)(x)=e_j(x)\frac{1-e_j(q_n\alpha)}{1-e_j(\alpha)}.
$$
For $|j|\geq q_n$, we have
$$
\Big|a_j\frac{1-e_j(q_n\alpha)}{1-e_j(\alpha)}\Big|\leq |j|^{-2-\epsilon}q_n.
$$
Therefore (for every $n\in \N$),
\beq\label{wazna2}
\Big|\sum_{|j|\geq q_n}a_j S_{q_n}(e_j)(x)\Big|\leq q_n\sum_{|j|\geq q_n} |j|^{-2-\epsilon}<q_n^{-\epsilon}.
\eeq
Moreover, notice that for $0<|j|<q_n$, using $\|jq_n\alpha\|<|j|/q_{n+1}$ and $|1-e(|j|)|\sim \|j\alpha\|$, we obtain
$$
\Big|a_j\frac{1-e_j(q_n\alpha)}{1-e_j(\alpha)}\Big|\ll q_{n+1}^{-1}|j|^{-1-\epsilon}\|j\alpha\|^{-1}.
$$
Hence,
\beq\label{wazna3}
\Big|\sum_{|j|< q_n}a_j S_{q_n}(e_j)(x)\Big|\leq q_{n+1}^{-1}\sum_{|j|< q_n} |j|^{-1-\epsilon}\|j\alpha\|^{-1}.
\eeq

Since $\|q_n\alpha\|<q_{n+1}^{-1}<q_n^{-1}$, for $|k|\leq q_n^{\epsilon^2/2}$, we have $\|kq_n\alpha\|\leq q_n^{-1+\epsilon/4}$.
So, by the cocycle identity for $S_r(\phi)$, the lemma follows by showing that there exists a subsequence $\{q_{n_m}\}$ satisfying
\beq\label{wazna}
\sup_{x\in \T}|S_{q_{n_m}}(\phi)(x)|\ll q_{n_m}^{-\epsilon/8}.
\eeq
We consider two cases:\\
\textbf{A.} There exists a subsequence $\{q_{n_m}\}$ such that $q_{n_m+1}>q_{n_m}^{1+\epsilon/4}$. In this case we will show that \eqref{wazna} holds for $\{q_{n_m}\}$.\\
\textbf{B}. For all (sufficiently large) $n$, $q_{n+1}\leq q_n^{1+\epsilon/4}$.

Note, moreover, that $\sup_{|j|<q_n}\|j\alpha\|^{-1}<2q_n$. Using this, if we are in case \textbf{A.}, along the sequence $\{q_{n_m}\}$, \eqref{wazna3} is bounded above by
$$
q_{n_m+1}^{-1}\sum_{|j|< q_{n_m}} |j|^{-1-\epsilon}\|j\alpha\|^{-1}\ll q_{n_m}^{-\epsilon/4},
$$
which, together with \eqref{wazna2}, finishes the proof if \textbf{A} holds.

If \textbf{B.} holds, we will show that \eqref{wazna} holds for $n_m=m$. Let $0<k\leq n$. Notice that  \eqref{wazna3} is bounded above by
$$
q_{n+1}^{-1}\sum_{|j|< q_k} |j|^{-1-\epsilon}\|j\alpha\|^{-1}+q_{n+1}^{-1}\sum_{q_k\leq |j|< q_n} |j|^{-1-\epsilon}\|j\alpha\|^{-1}.
$$
For the first summand above, we use that $\sup_{|j|<q_k}\|j\alpha\|^{-1}<2q_k$ and therefore,
$$
q_{n+1}^{-1}\sum_{|j|< q_k} |j|^{-1-\epsilon}\|j\alpha\|^{-1}\ll \frac{q_k}{q_{n+1}}.
$$
Moreover, since $\sup_{|j|<q_n}\|j\alpha\|^{-1}<2q_n<2q_{n+1}$,
$$
q_{n+1}^{-1}\sum_{q_k\leq |j|< q_n} |j|^{-1-\epsilon}\|j\alpha\|^{-1}\ll \sum_{q_k\leq |j|<q_n} |j|^{-1-\epsilon}\ll q_k^{-\epsilon}.
$$
Putting the above bounds together, we get that \eqref{wazna3} is
$$
\ll \frac{q_k}{q_{n+1}}+q_k^{-\epsilon}.
$$
We now choose $k$ so that $q_k\in [q_n^{1/4}, q_n^{1/2}]$. Notice that this is possible since we are in Case~\textbf{B.} (and so, for every sufficiently large $a$, there exists $k_0\in \N$ such that $q_{k_0}\in[a,a^2]$). Then \eqref{wazna3} is
$\ll q_n^{-\epsilon/4}$.
Using this and \eqref{wazna2}, we get that \eqref{wazna} holds for $n_m=m$.
This finishes the proof.
\end{proof}

\subsection{Smooth flows on $\T^2$ and Rokhlin extensions}
Using the same arguments as in Section~\ref{sub:anz}, we can provide some new instances of M\"obius disjointness. Recall that smooth time changes of the linear flow on $\T^2$ are given by (here $\alpha\in\T$ is irrational) $\frac{dx}{dt}=\frac{\alpha}{F(x,y)}$, $\frac{dy}{dt}=\frac1{F(x,y)}$ for a smooth, positive function $F:\T^2\to\R$. They have a special representation as a (special) flow $T^f=(T^f_t)_{t\in\R}$ over the irrational rotation $Tx=x+\alpha$ and under a smooth function $f:\T\to\R$ satisfying (for simplicity) $\int f\,d Leb=1$ (in fact, each smooth area-preserving flow on $\T^2$ has such a representation whenever the flow has no fixed points). The space $\T^f$ has a natural metric $D$ making it a compact metric space, and
\beq\label{mala}
D(T^f_t(x,y),(x,y))\leq |t|\text{ whenever $t$ is small enough.}\eeq
We have
$$
D(T^f_{q_n}(x,y),(x,y))=
D(T^f_{q_n-S_{q_n}(f)(x)}(T^f_{S_{q_n}(f)(x)}(x,y)),(x,y))\leq
$$$$
D(T^f_{q_n-S_{q_n}(f)(x)}(T^{q_n}x,y),(T^{q_n}x,y))+D((T^{q_n}x,y),
(x,y)).$$
Using \eqref{mala} and \eqref{wazna}, we obtain PR rigidity condition (along the rigidity time $\{q_{n_m}\}$ for the special flow $T^f$). We have proved:

\begin{Cor}\label{cor:spflow}
Time-1 maps of the special flows $T^f$ over an irrational rotation $Tx=x+\alpha$ and under a $C^{2+\epsilon}$ roof functions $f$ are M\"obius disjoint.\end{Cor}

Another class of examples is given by Rokhlin extensions. Here, we fix an irrational rotation $Tx=x+\alpha$ and a smooth zero mean $f:\T\to\R$. Assume that $\mathcal{R}=(R_t)_{t\in \R}$ is a (measurable) flow on $(Y,\mathcal{C},\kappa)$. Then, on the product space $X\times Y$, we have the corresponding Rokhlin extension $T_{f,\mathcal{R}}$ given by
$$
T_{f,\mathcal{R}}(x,y)=(Tx,R_{f(x)}(y)).$$
(Note that Anzai skew products are special cases of this construction in which $Y=\T$ and $\mathcal{R}$ stands for the linear flow on $\T$.)
If $\mathcal{R}$ is a smooth flow on a compact manifold $Y$ with a metric $\rho$ then (for the product metric $D$ on $X\times Y$) we have
$$
D(T^{q_n}_{f,\mathcal{R}}(x,y)(x,y))=
d(T^{q_n}x,x)+\rho(R_{S_{q_n}(f)(x)}(y),y),$$
if the flow $\mathcal{R}$ is Lipschitz continuous, then~\eqref{wazna} will again be satisfied and we obtain the following:

\begin{Cor}\label{rokext} Assume that $f\in C^{2+\epsilon}(\T)$. Then for
{\bf all} Lipschitz continuous $\mathcal{R}$,  Rokhlin extensions $T_{f,\mathcal{R}}$
are M\"obius disjoint.\end{Cor}

Notice that one can take $\mathcal{R}=(h_t)_{t\in\R}$ to be the horocycle flow. In this case the fiber dynamics is very different from the case of Anzai skew-products (mixing with countable Lebesgue spectrum) and, by the above corollary, M\"obius disjointness still holds.

\begin{Remark} \em If the sequence $\{q_n\}$ of denominators of $\alpha$ has bounded prime volume (that BPV rigidity holds) then the above corollary is true for {\bf all continuous} flows $\mathcal{R}$. In fact, the M\"obius disjointness holds for uniquely ergodic models for {\bf all measurable} flows $\mathcal{R}$,  to be compared with a result of \cite{Ku-Le}.\end{Remark}

\subsection{Proof of Theorem~\ref{th:main1}}\label{s:main1}
\begin{proof}
To simplify notation, we will assume that $x$ is generic for $\nu$ which yields a rigid system with a rigidity time $\{q_n\}$ along which  either BPV or PR rigidity holds (see Definition~\ref{def:good}).

Fix a continuous $f : X \rightarrow \R$, where in case PR rigidity is satisfied, we assume additionally that $f\in\mathcal{F}$. Select $L_n\to\infty$ slowly enough to have
$$
\sum_{j=-L_n+1}^{L_n-1}\|f\circ T^{jq_n}-f\|^2_{L^2(\nu)}\to 0.$$
Note that such a sequence obviously exists when BPV rigidity (along $\{q_n\}$) is satisfied, while in the case of PR rigidity, we simply take $L_n=q_n^{\delta}$.
Fix $\vep>0$ (sufficiently small). Then, for $n$ large enough (which we fix)
\beq\label{rig2}
\int_X\sum_{j=-L_n+1}^{L_n-1}\left| f\circ T^{jq_n}-f\right|^2\,d\nu<\vep.\eeq
Since $x$ is generic for $\nu$, by \eqref{rig2}, we obtain
$$
\lim_{M\to\infty}\frac1M\sum_{m\leq M}
\left( \sum_{j=-L_n+1}^{L_n-1}
\left| f(T^{jq_n+m}x)-f(T^mx)\right|^2\right)<\vep.$$
Hence, for some $M_0=M_0(\vep)$ and every $M>M_0$, we have
\beq\label{rig3}
\frac1M\sum_{m\leq M}
\left( \sum_{j=-L_n+1}^{L_n-1}
\left| f(T^{jq_n+m}x)-f(T^mx)\right|^2\right)<\vep.\eeq
We say that $m\leq M$ is {\em good} if
\beq\label{rig01}
\sum_{j=-L_n+1}^{L_n-1}
\left| f(T^{jq_n+m}x)-f(T^mx)\right|^2<\vep^{1/2}.\eeq
Then by Markov's inequality,
\beq\label{rig4}
|\{m\leq M:\: m\text{ is good}\}|>(1-\vep^{1/2})M.\eeq
We also assume that~\eqref{eq:klucz} holds for $M$ (with $\vep$ replaced by $\vep/\|f\|_\infty$). Moreover, as $M$ is arbitrarily large compared to $L_nq_n$, no harm to assume that $z=0$; indeed, otherwise replace intervals $[jL_nq_n,(j+1)L_nq_n)$ by $[z+jL_nq_n,z+(j+1)L_nq_n)$ in the reasoning below.
So now, we write $[0,M]=\bigcup_{j=0}^{M/(L_nq_n)}[jL_nq_n,(j+1)L_nq_n)$
(we do not pay attention to the last interval as $M$ is arbitrarily large with respect to $L_nq_n$). We say that an interval $[jL_nq_n,(j+1)L_nq_n)$ is {\em good} if the number of good $m$ in it is at least $(1-\vep^{1/4})L_nq_n$. It follows that
\beq\label{rig5}
\left|\left\{j\leq \frac M{L_nq_n}:\: [jL_nq_n,(j+1)L_nq_n)\text{ is good}\right\}\right|\geq (1-\vep^{1/4})\frac M{L_nq_n}.\eeq
Indeed, if $K$ denotes the number of good intervals, then by \eqref{rig4},
$$
(1-\vep^{1/2})M\leq \Big(\frac{M}{L_nq_n}-K\Big)(1-\vep^{1/4})L_nq_n+K L_nq_n,
$$
which implies that $(\epsilon^{1/4}-\epsilon^{1/2})\frac{M}{L_nq_n}\leq \epsilon^{1/4}K$ and this proves \eqref{rig5}.

Now, take $[jL_nq_n,(j+1)L_nq_n)$ and assume that it is good. We will consider numbers in this interval mod~$q_n$. We say that $a\in\{0,1,\ldots, q_n-1\}$ is {\rm good} if there exists $m=m_a\in [jL_nq_n,(j+1)L_nq_n)$ which is good and $m\equiv a$ mod~$q_n$. Note that
\beq\label{rig6}
|\{a\in\{0,1,\ldots,q_n-1\}: a\text{ is good}\}|\geq (1-\vep^{1/4})q_n.\eeq
Indeed, if $a$ is bad, then it produces $L_n$ of bad $m$. Note also that whenever $a$ is good (and $m_a\in [jL_nq_n,(j+1)L_nq_n)$ is good with $m_a\equiv a \mod q_n$) then for each $m_1,m_2\in[jL_nq_n,(j+1)L_nq_n)$, with $m_i\equiv a \text{ mod }q_n,\;i=1,2$, by~\eqref{rig01} for $m_a$,  we have
\beq\label{rig7}
|f(T^{m_1}x)-f(T^{m_2}x)|\leq |f(T^{m_1}x)-f(T^{m_a}x)|+|f(T^{m_a}x)-f(T^{m_2}x)|\leq 2 \vep^{1/4}.\eeq
We fix $[jL_nq_n,(j+1)L_nq_n)$ which is good and we evaluate the relevant part of the sum $\frac1M\sum_{m\leq M}f(T^mx)\mob(m)$ using \eqref{rig6} and then \eqref{rig7}:
$$
\left|\sum_{m\in [jL_nq_n,(j+1)L_nq_n)}f(T^mx)\mob(m)\right|\leq$$$$\sum_{a=0\atop a\text{ good}}^{q_n-1}\left|\sum_{m\in [jL_nq_n,(j+1)L_nq_n)\atop m\equiv a\text{ mod }q_n}f(T^mx)\mob(m)\right|+\|f\|_\infty\cdot \vep^{1/4}L_nq_n\leq$$
$$
\|f\|_\infty\sum_{a=0\atop a\text{ good}}^{q_n-1}\left|\sum_{m\in [jL_nq_n,(j+1)L_nq_n)\atop m\equiv a\text{ mod }q_n}\mob(m)\right|+2\vep^{1/4}L_nq_n+\|f\|_\infty\cdot \vep^{1/4}L_nq_n.$$
We have (in view of \eqref{rig5} and the estimate above)
$$
\left|\sum_{m\leq M}f(T^mx)\mob(m)\right|=
\left|\sum_{j=0}^{M/(L_nq_n)}\sum_{m\in [jL_nq_n,(j+1)L_nq_n)}f(T^mx)\mob(m) \right|\leq$$
$$
\sum_{j=0\atop j\text{ good}}^{M/(L_nq_n)}\left|\sum_{m\in [jL_nq_n,(j+1)L_nq_n)}f(T^mx)\mob(m) \right|+\vep^{1/2}\frac M{L_nq_n}\cdot \|f\|_\infty L_nq_n\leq$$$$
\sum_{j=0\atop j\text{ good}}^{M/(L_nq_n)} \left(\|f\|_\infty \sum_{a=0\atop a\text{ good}}^{q_n-1}\left|\sum_{m\in [jL_nq_n,(j+1)L_nq_n)\atop m\equiv a\text{ mod }q_n}\mob(m)\right|+2\vep^{1/4}L_nq_n+\|f\|_\infty\cdot \vep^{1/4}L_nq_n\right) +$$
$$
\vep^{1/2}M\|f\|_\infty=
\|f\|_\infty\sum_{j=0\atop j\text{ good}}^{M/(L_nq_n)} \sum_{a=0\atop a\text{ good}}^{q_n-1}\left|\sum_{m\in [jL_nq_n,(j+1)L_nq_n)\atop m\equiv a\text{ mod }q_n}\mob(m)\right|+ {\rm O}(\vep^{1/7}M).$$
The last expression is ${\rm O}(\vep^{1/10} M)$ since we have assumed~\eqref{eq:klucz} in Corollary~\ref{cor1} to hold for $z=0$. Indeed, when BPV rigidity holds, we can apply Corollary~\ref{cor1}, since $L_n\to\infty$, so $\log\log L_n\to\infty$ and the prime volume of $\{q_n\}$ is bounded, while if PR rigidity is satisfied, we first use the second statement in Lemma~\ref{lem:qa} to see that Corollary~\ref{cor1} is applicable (recall that $L_n=q_n^\delta$). The result follows.
\end{proof}

\begin{Remark}\label{r:dodat}\em Note that the argument used at the beginning of the proof gives the following. If $T$ is uniquely ergodic and
$\sum_{j=-L_n}^{L_n}\|f\circ T^{jq_m}-f\|_{L^2(\nu)}<\vep$ then for all $M\geq M_0$ we have
$$\left\|\frac1M\sum_{m\leq M}\sum_{j=-L_n}^{L_n}|f\circ T^{jq_n+m}-f\circ T^m|\right\|_{C(X)}<\vep.$$\end{Remark}

\section{Systems with countably many ergodic rigid measures. Proof of Theorem~\ref{thm:mainE}} \label{czesc4}

Following \cite{Ab-Ku-Le-Ru1} and \cite{Go-Le-Ru}, a dynamical system $(X,T)$ is said to satisfy the strong MOMO property if for each increasing sequence $(b_k)$ of natural numbers,
$b_{k+1}-b_k\to\infty$, and each $f\in C(X)$, we have
\beq\label{momo1}
\frac1{b_{K}}\sum_{k<K}\left\|\sum_{b_k\leq n<b_{k+1}}\mob(n)f\circ T^{n}\right\|_{C(X)}\to 0\text{ when }K\to\infty.\eeq
Even though the strong MOMO property looks stronger than the original M\"obius disjointness, as proved in \cite{Ab-Ku-Le-Ru1}, Sarnak's conjecture is equivalent to the fact that all zero entropy systems satisfy the strong MOMO property.
Moreover, the strong MOMO property implies uniform (in $x\in X$) convergence in~\eqref{sar1}.

Given $x\in X$, by $V(x)\subset M(X,T)$ we denote the set of  measures for which $x$ is quasi-generic.

In what follows we need an extension of the main result from \cite{Ab-Ku-Le-Ru1}.

Let $\left(\left(Z_i,\cd_i,\kappa_i,R_i\right)\right)_{i\geq1}$, be a sequence of ergodic dynamical systems (this means that we admit a repetition of the same dynamical system infinitely many times). Consider the following three conditions:

\vspace{1ex}

{\bf (PF1)} For each $i\geq1$, there is a topological system $(Y_i,S_i)$ satisfying the strong MOMO property and for some $\mu_i\in M^e(Y_i,S_i)$, the measure-theoretic systems $(Z_i,\cd_i,\kappa_i,R_i)$ and $(Y_i,\cb(Y_i),\mu_i,S_i)$ are (measure-theoretically) isomorphic.

\vspace{1ex}

{\bf (PF2)} For each topological system $(X,T)$ and $x\in X$ satisfying:
\begin{itemize}
\item $V(x)\subset\left\{\sum_{j\geq1}\alpha_{j}\mu_{j}
    :\:\mu_j\in M(X,T),\alpha_{j}\geq0\text{ for }j\geq1,\; \sum_{j\geq1}\alpha_{j}=1\right\}$,
\item the  (measure-theoretic) systems $(X,\cb(X),\mu_{j},T)$ and $(Z_{i_j},\cd_{i_j},\kappa_{i_j},R_{i_j})$ (for some $i_j\geq1$) are measure-theoretically isomorphic for each $j\geq1$,
    \end{itemize} we have that the point $x$ satisfies the $\mob$-Sarnak property: $$\lim_{N\to\infty}\frac1N\sum_{n\leq N}f(T^nx)\mob(n)=0.$$

\vspace{1ex}

{\bf (PF3)} For each topological system $(Y,S)$ for which $M^e(Y,S)=\{\nu_j:\:j\geq1\}$ and, for each $j\geq1$, the (measure-theoretic) systems $(Y,\cb(Y),\nu_j,S)$ and $(Z_{i_j},\cd_{i_j},\kappa_{i_j},R_{i_j})$ (for some $i_j\geq1$)  are measure-theoretically isomorphic, we have that $(Y,S)$ satisfies the strong MOMO property.

\begin{Th}\label{t:main2} Conditions {\bf (PF1)}, {\bf (PF2)} and {\bf (PF3)} are equivalent.\end{Th}
Theorem~\ref{t:main2}, which is of independent interest, will be proved in the appendix in a more general setting. It is an extension of Main Theorem in \cite{Ab-Ku-Le-Ru1} in which only finitely many measures are used, all of them giving rise to the same measure-theoretic dynamical system.

If $R$ is an ergodic automorphism on a standard probability space $\zdr$  which is rigid then it remains rigid in every of its uniquely ergodic models. Therefore, such a model is M\"obius disjoint whenever we have BPV rigidity or PR rigidity (by Theorem~\ref{th:main1}). In fact, more is true.

\begin{Cor}\label{c:momo1}
Assume that $(Y,S)$ is uniquely ergodic, with the unique invariant measure $\nu$ which yields either BPV rigidity or PR rigidity. Then $(Y,S)$ satisfies the strong MOMO property.
\end{Cor}
\begin{proof}I. Consider first BPV rigidity. We check {\bf (PF2)}. So, assume that $(X,T)$ is a topological system and $x\in X$ satisfies \\$V(x)\subset \left\{ \sum_{j\geq1}\alpha_j\mu_j:\:\alpha_j\geq0,\;\sum_{j\geq1}\alpha_j=1
\right\}$, where $\mu_j$ yields a system \\ measure-theoretically isomorphic to $(Y,\cb(Y),\nu,S)$ (we use Theorem~\ref{t:main2} with the constant sequence equal to $(Y,\cb(Y),\nu,S)$). Since all measures $\mu_j$ yield the same  (up to isomorphism) system, all the measures $\sum_{j\geq 1}\alpha_j\mu_j$ yield systems having a common rigidity sequence. Hence, $x$  is completely BPV rigid. It follows from Theorem~\ref{th:main1} that $x$ satisfies the $\mob$-Sarnak property and {\bf (PF2)} holds. Hence {\bf (PF3)} holds which completes the proof.

II. Assume that we have PR rigidity. Taking into account the definition of strong MOMO property, we need to estimate \\
$\frac1{b_K}\sum_{k<K}\left\|\sum_{b_k\leq m<b_{k+1}}\mu(m)f\circ T^m\right\|_{C(X)}$.  We essentially repeat the proof of Theorem~\ref{th:main1}. In view of Remark~\ref{r:dodat}, for $M\geq M_0$,
$$
\left\|\frac1M\sum_{m\leq M}\sum_{j=-L_n}^{L_n}|f\circ T^{jq_n+m}-f\circ T^m|\right\|_{C(X)}<\vep.$$
Assume for a while that the numbers $b_k$ are all multiples of $L_nq_n$. Then the proof of Theorem~\ref{th:main1} readily repeats (with $M=b_K$). If not, remembering that $b_{k+1}-b_k\to\infty$, we can choose a sequence $\{b'_k\}$, so that: $b'_k$ are multiples of $L_nq_n$ (for $k$ sufficiently large)
and $|b_k-b'_k|\leq L_nq_n$.  Now,
$$
\left|\sum_{k<K}\left\|\sum_{b_k\leq m<b_{k+1}}\mu(m)f\circ T^m\right\|_{C(X)} - \sum_{k<K}\left\|\sum_{b'_k\leq m<b'_{k+1}}\mu(m)f\circ T^m\right\|_{C(X)}\right|=$$$$ o(b_K)+ 2\|f\|_{C(X)}KL_nq_n,$$
which is sufficient for our proof.
\end{proof}

In fact, still more is true (we stress that the automorphisms in the sequence $(Z_i,\cd_i,\kappa_i,R_i)$ below need not have a common rigid sequence).

\begin{Cor}\label{c:momo2} Assume that $\left(\left(Z_i,\cd_i,\kappa_i,R_i\right)\right)_{i\geq1}$ is a sequence of ergodic automorphisms each of which is BPV or PR rigid.
Assume that we have a topological system $(Y,S)$ for which $M^e(Y,S)=\{\nu_j:\:j\geq1\}$ and, for each $j\geq1$, the (measure-theoretic) systems $(Y,\cb(Y),\nu_j,S)$ and $(Z_{i_j},\cd_{i_j},\kappa_{i_j},R_{i_j})$ (for some $i_j\geq1$)  are measure-theoretically isomorphic. Then $(Y,S)$ satisfies the strong MOMO property.\end{Cor}
\begin{proof} It follows from Corollary~\ref{c:momo1} that {\bf (PF1)} is satisfied. By Theorem~\ref{t:main2}, {\bf (PF3)} holds, so the result follows.\end{proof}

{\em Proof of Theorem~\ref{thm:mainE}}.
The result follows directly from
Corollary~\ref{c:momo2}.

\section{Appendix: Proof of Theorem~\ref{t:main2}}

The proof of Theorem~\ref{t:main2} does not use any special property of $\mob$ except of boundedness. Therefore, in what follows we consider the notion of strong $\bfu$-OMO and Theorem~\ref{t:main2} in which the M\"obius function $\mob$ has been replaced
by $\bfu:\N\to\C$ any bounded arithmetic function.

In order to prove the equivalence of  conditions {\bf (PF1)}-{\bf (PF3)}, we need the following two auxiliary lemmas. The proofs follow the same lines
as the proof of Lemma~17 in \cite{Ab-Ku-Le-Ru1} and the arguments just before this lemma.

\begin{Lemma}\label{lem:1}
Let $(X,T)$ be a topological system and $\nu\in M(X,T)$. Suppose that $C\subset X$ is a compact subset such that $\nu(C)>1-\vep^2$ for some
$0<\vep<1$. Then, for every $L\geq 1$, the set
\[B_L(C,\vep):=\Big\{x\in C:\frac{1}{L}\sum_{l<L}\chi_{C}(T^lx)>1-\vep\Big\}\]
is compact and $\nu(B_L(C,\vep))>1-\vep$.
\end{Lemma}

\begin{Lemma}\label{lem:2}
Let $(X,T)$ be a topological system. Let $x\in X$ and $C\subset X$ be a compact subset.
Suppose that $\nu\in M(X,T)$ and $(N_i)_{i\geq 1}$ is an increasing sequence of natural numbers
such that
\[\frac{1}{N_i}\sum_{n<N_i}\delta_{T^nx}\to\nu\]
weakly. Then, for every $\eta>0$, we have
\[\limsup_{i\to\infty}\frac{1}{N_i}\#\big\{0\leq n<N:d(T^nx,C)\geq \eta\big\}\leq \nu(X\setminus C).\]
\end{Lemma}

\begin{proof} {\em of Theorem~\ref{t:main2}} {\bf (PF3)} $\Rightarrow$ {\bf (PF1)} To obtain {\bf (PF1)}, for each $i\geq 1$, using Jewett-Krieger theorem, choose  $(Y_i,S_i)$ being a uniquely ergodic model  of $(Z_i,\mathcal{D}_i,\kappa_i,R_i)$. Then apply apply {\bf (PF3)} for $(Y_i,S_i)$.

{\bf (PF1)} $\Rightarrow$ {\bf (PF2)} (is a modification of the proof from \cite{Ab-Ku-Le-Ru1}).
Suppose contrary to our claim that there exist a topological system $(X,T)$, $x\in X$, a continuous function $f:X\to\C$ with $\|f\|_{\infty}\leq 1$ and $0<\vep<1/2$ such that:
\begin{itemize}
\item[$(i)$] $V(x)\subset\left\{\sum_{j\geq1}\alpha_{j}\mu_{j}
    :\:\mu_j\in M(X,T),\alpha_{j}\geq0\text{ for }j\geq1,\; \sum_{j\geq1}\alpha_{j}=1\right\}$;
\item[$(ii)$] for each $j\geq1$  there exists $i_j\geq 1$ such that the  measure-theoretic systems $(X,\cb(X),\nu_{j},T)$ and $(Z_{i_j},\cd_{i_j},\kappa_{i_j},R_{i_j})$  are measure-theoretically isomorphic;
\item[$(iii)$] $\limsup_{N\to\infty}\Big|\frac1N\sum_{n\leq N}f(T^nx)\bfu(n)\Big|>7\vep>0$.
    \end{itemize}
Therefore, we can find an increasing sequence $(N_i)_{i\geq 1}$ such that
\begin{align}
&\frac{1}{N_i}\sum_{n<N_i}\delta_{T^nx}\to\nu\in V(x)\quad\text{weakly as }\quad i\to\infty;\nonumber\\
&\Big|\frac{1}{N_i}\sum_{n\leq N_i}f(T^nx)\bfu(n)\Big|>7\vep\quad \text{for all}\quad i\geq 1. \label{eq:6vep}
\end{align}
By condition $(i)$, $\nu=\sum_{j\geq1}\alpha_{j}\nu_{j}$ for a sequence $(\alpha_j)_{j\geq 1}$ of positive numbers with $\sum_{j\geq1}\alpha_{j}=1$ so that $(\nu_j)_{j\geq 1}$ is a sequence of different $T$-invariant ergodic measures satisfying $(ii)$.
Choose a natural number $t$ such that
\begin{equation}
\sum_{j>t}\alpha_j<\vep/2.
\end{equation}
In view of $(ii)$ and \textbf{(PF1)}, for every $1\leq j\leq t$, there exists a topological system $(Y_j,S_j)$ satisfying the strong $\bfu$-OMO property and an invariant ergodic measure $\mu_j\in M^e(Y_j,S_j)$ so that the measure-theoretic systems $(X,\cb(X),\nu_j,T)$ and $(Y_j,\cb(Y_j),\mu_j,S_j)$ are measure-theoretically isomorphic.

Since $\nu_j$ for $1\leq j\leq t$ are different ergodic $T$-invariant measures, we can find $T$-invariant Borel subsets $X_j\subset X$ for $1\leq j\leq t$ with $\nu_j(X_j)=1$, and measure-theoretic isomorphisms $\phi_j:(X_j,\cb(X_j),\nu_j,T)\to (Y_j,\cb(Y_j),\mu_j,S_j)$. By Lusin's theorem, for every $1\leq j\leq t$ there exists a compact subset $W_j\subset X_j$ such that $\nu_j(W_j)>1-\vep^4$ and the restriction  $\phi_j:W_j\to \phi_j(W_j)$ is a homeomorphism.

In view of Lemma~\ref{lem:1}, for every $L\geq 1$ and $1\leq j\leq t$ the set $B_L(W_j,\vep^2)$ is compact and \[\nu_j(B_L(W_j,\vep^2))>1-\vep^2\geq 1-\vep/2.\]
Let $B_L:=\bigcup_{1\leq j\leq t}B_L(W_j,\vep^2)$. Then $B_L$ is a compact subset of $X$ such that $\nu_j(B_L)>1-\vep/2$
for every $1\leq j\leq t$. Since $\nu=\sum_{j\geq 1}\alpha_j\nu_j$ and $\sum_{j>t}\alpha_k<\vep/2$ we have
\[\nu(B_L)>1-\vep.\]
Therefore, by Lemma~\ref{lem:2}, for every $\eta>0$ and $L\geq 1$, we have
\begin{equation}\label{neq:limsup}
\limsup_{i\to\infty}\frac{1}{N_i}\#\big\{0\leq n<N_i:d(T^nx,B_L)\geq \eta\big\}<\vep.
\end{equation}

For every $L\geq 1$ let $\eta(L)$ be a positive number such that
\[d(w,w')<\eta(L)\Rightarrow \forall_{0\leq l<L}\;\;d(T^lw,T^lw')<\vep.\]

Fix an increasing sequence $(L_i)_{i\geq 1}$ of natural numbers. In view of \eqref{neq:limsup},
there exists $(M_i)_{i\geq 1}$ a subsequence of $(N_i)_{i\geq 1}$ such that $M_i-M_{i-1}\to+\infty$,
\begin{equation}
L_i<\vep M_i\quad\text{for every}\quad i\geq 1,
\end{equation}
and, setting $M_0=0$, we
have
\begin{equation}\label{neq:mimi}
\frac{1}{M_i-M_{i-1}}\#\big\{M_{i-1}\leq n<M_i:d(T^nx,B_{L_i})\geq \eta(L_i)\big\}<\vep\ \text{for}\ i\geq 1.
\end{equation}

A natural number $b$ is said to be \emph{good} if, for some $i\leq 1$, we have $M_{i-1}\leq b<M_i$ and $d(T^bx,B_{L_i})< \eta(L_i)$.

Let us define the increasing sequence $(b_i)_{i\geq1}$ of good numbers inductively in the following way: set $b_0=0$ and let $b_1\geq 1$ be the smallest good number. Assume that $M_{i-1}\leq b_k<M_i$ is defined for some $k\geq 1$. Then we define $b_{k+1}$ as the smallest good number $b\geq b_k+L_i$. As $b_{k+1}-b_k\geq L_i$ and $L_i\to+\infty$, we have $b_{k+1}-b_k\to+\infty$.

For every $1\leq j\leq t$ the function $f\circ \phi^{-1}_j$ is continuous on the compact set $\phi_j(W_j)\subset Y_j$. Then, by Tietze extension theorem, there is a continuous function
$g_j:Y_j\to\C$ such that $\|g_j\|_{\infty}\leq\|f\|_{\infty}\leq 1$ and $g_j=f\circ \phi_j^{-1}$ on $W_j$. Since $\phi_j$ establishes a measure-theoretical isomorphism between $(X,\cb(X),\nu_j,T)$ and $(Y_j,\cb(Y_j),\mu_j,S_j)$,
we have
\begin{equation}
w\in W_j\cap T^{-l}W_j\ \Rightarrow\ f(T^lw)=g_j(\phi_j(T^lw))=g_j(S_j^l\phi_j(w)).
\end{equation}

Since $b_k$ is good, there exist $1\leq j_k\leq t$ and $x_k\in B_{L_i}(W_{j_k},\vep^2)$ such that $d(T^{b_k}x,x_k)<\eta(L_i)$. It follows that
\[|f(T^{b_k+l}x)-f(T^lx_k)|<\vep\quad\text{for}\quad 1\leq l<L_i.\]
Since $x_k\in B_{L_i}(W_{j_k},\vep^2)$, the average frequency of the orbit $\{T^lx_k:0\leq l<L_i\}$ in $W_{j_k}$ is at least $1-\vep^2$. Therefore, by setting $y_k:=\phi(x_k)\in Y_{j_k}$,
we obtain
\[\sum_{l<L_i}|f(T^{b_k+l}x)-g_{j_k}(S_{j_k}^ly_k)|\leq 2\vep^2L_i\|f\|_{\infty}+\vep L_i.\]
In view of \eqref{neq:mimi}, it follows that
\begin{align*}
\sum_{b_k\leq s<b_{k+1}}&|f(T^{s}x)-g_{j_k}(S_{j_k}^s(S_{j_k}^{-b_k}y_k))|\\
&\leq \vep(b_{k+1}-b_k)(\|f\|_{\infty}+1)+2\|f\|_{\infty}\#\{b_k\leq s<b_{k+1}; s\text{ is not good}\}\\
&\leq 4\vep(b_{k+1}-b_k).
\end{align*}
Hence, for every $K\geq 1$, we have
\[\frac{1}{b_K}\sum_{s<b_{K}}|f(T^{s}x)-g_{j_k}(S_{j_k}^s(S_{j_k}^{-b_k}y_k))|\leq 4\vep.\]
By the strong $\bfu$-OMO property of the topological systems $(Y_j,S_j)$ for $1\leq j\leq t$, we have
\[\lim_{K\to\infty}\frac{1}{b_K}\sum_{k<K}\left|\sum_{b_k\leq s<b_{k+1}}g_{j_k}(S_{j_k}^s(S_{j_k}^{-b_k}y_k))\bfu(s)\right|=0.\]
It follows that
\[\limsup_{K\to\infty}\frac{1}{b_K}\left|\sum_{s<b_{K}}f(T^{s}x)\bfu(s)\right|\leq 4\vep.\]

For every $i\geq 1$ denote by $K_i$ the largest $K$ such that $b_K\leq M_i$. Then
\[M_i-b_{K_i}\leq L_i+\#\{n<M_i:n\text{ is not good}\}<2\vep M_i\] and
\begin{align*}
\frac{1}{M_i}\left|\sum_{s<M_i}f(T^{s}x)\bfu(s)\right|
&\leq \frac{1}{b_{K_i}}\left|\sum_{s<b_{K_i}}f(T^{s}x)\bfu(s)\right|+
\frac{M_i-b_{K_i}}{M_i}\\
&\leq \frac{1}{b_{_i}K}\left|\sum_{s<b_{K_i}}f(T^{s}x)\bfu(s)\right|+2\vep.
\end{align*}
Hence
\[\limsup_{i\to\infty}\frac{1}{M_i}\left|\sum_{s<M_i}f(T^{s}x)\bfu(s)\right|\leq 6\vep,\]
contrary to \eqref{eq:6vep} (as $(M_i)$ is a subsequence of $(N_i)$).

{\bf (PF2)} $\Rightarrow$ {\bf (PF3)} We fix $f\in C(Y)$,  $(b_k)$
satisfying $b_{k+1}-b_k\to\infty$ and $(y_k)\subset Y$. Let $\Sigma_3$ be the group of roots of unity of degree~3. Choose $e_k\in\Sigma_3$ so that
$$
e_k\left(\sum_{b_k\leq n<b_{k+1}}f(S^n(S^{-b_k}y_k))\bfu(n)\right)$$
is in the cone $\{0\}\cup\{z\in\C:\:{\rm arg}(z)\in[-\pi/3,\pi/3]\}$. Set
$$
x=(x_n)\in (Y\times\Sigma_3)^{\N},\text{ where }x_n=(S^{n-b_k}y_k,e_k)\text{ for }b_k\leq n<b_{k+1}$$
and
$$
X_x:=\ov{\{T^nx:n\geq0\}},$$
where $T$ stands for the shift on $(Y\times\Sigma_3)^{\N}$.
We need to determine all ergodic ($T$-invariant) measures on $X_x$. Recall that each such measure has a generic point. On the other hand, the analysis done in the proof of Corollary~9 \cite{Ab-Ku-Le-Ru1}, shows that if $z\in X_x$ is a quasi-generic point for a measure $\rho\in M(X_x,T)$, then the projection on the first coordinate in $X_x\subset (Y\times \Sigma_3)^{\N}$-coordinate $\rho$-a.e. intertwines $T$ and $S\times Id$. It follows that if $\rho$ is ergodic, then the projection $\rho^{(1)}$ of $\rho$ is an $S\times Id$-invariant measure which is ergodic (for this map). By the classical disjointness result {\em ergodicity} $\perp$ {\em identity}, it follows that $\rho^{(1)}=\kappa\otimes \delta_a$, where $\kappa\in M^e(Y,S)$ and $a\in\Sigma_3$. Furthermore, $\rho$ is the image of $\rho^{(1)}$ via the map
$$
(Id\times Id)\times (S\times Id)\times(S^2\times Id)\times\ldots,$$
so $M^e(X_x,T)$ is countable and each ergodic measure yields a system which is measure-theoretically isomorphic to a system $(Y,\cb(Y),\nu,S)$ for some $\nu\in M^e(Y,S)$. We now apply {\bf (PF2)} to $x$ and the function
$$
((v_0,a_0), (v_1,a_1),\ldots)\mapsto a_0f(v_0)$$
in exactly the same manner as in \cite{Ab-Ku-Le-Ru1} to conclude.
\end{proof}

\end{document}